\documentclass[12pt,twoside]{amsart}
\usepackage{amssymb}
\usepackage{amscd}
\usepackage[abbrev,alphabetic]{amsrefs}
\usepackage{hyperref}

\usepackage{xypic}
\usepackage[all,cmtip]{xy}
\usepackage{tikz-cd}
\usepackage{caption}

\usepackage{url}
\RequirePackage{booktabs, multirow}
\RequirePackage{pgf}

\title[KV vanishing for log del Pezzo surfaces in char $p>0$]{On the Kawamata--Viehweg vanishing theorem for log del Pezzo surfaces in positive characteristic} 
\author[E. Arvidsson, F. Bernasconi, J. Lacini]{Emelie Arvidsson, Fabio Bernasconi, Justin Lacini} 

\address{School of Mathematics, Institute for Advanced Study, Princeton, NJ 08540, USA}
\email{emmiarwidsson@ias.edu}

\address{\'Ecole Polytechnique F\'ed\'erale de Lausanne, Chair of Algebraic Geometry
	(B\^atiment MA), MA B3 515, Station 8, CH-1015 Lausanne} 
\email{fabio.bernasconi@epfl.ch}

\address{University of Kansas,
	Department of Mathematics,
	643 Snow Hall, Lawrence, KS 66046, USA}
\email{jlacini@ku.edu}

\newtheorem{thm}{Theorem}[section]
\newtheorem{lem}[thm]{Lemma}
\newtheorem{cor}[thm]{Corollary}
\newtheorem{prop}[thm]{Proposition}

\theoremstyle{definition}

\newtheorem{definition}[thm]{Definition}

\newtheorem{remark}[thm]{Remark}

\makeatletter
  
  \@addtoreset{equation}{section}
  \makeatother
\newcommand{\Q}{\mathbb{Q}}
\newcommand{\Z}{\mathbb{Z}}

\newcommand{\mO}{\mathcal{O}}

\begin{document}

\begin{abstract}
We prove the Kawamata-Viehweg vanishing theorem for surfaces of del Pezzo type over perfect fields of positive characteristic $p>5$. As a consequence, we show that klt threefold singularities over a perfect base field of characteristic $p>5$ are rational. We show that these theorems are sharp by providing counterexamples in characteristic $5$.
\end{abstract}

\maketitle

\section{Introduction}

Fano varieties and their log generalisations are one of the basic building blocks of algebraic varieties according to the Minimal Model Program. In recent years a lot of progress has been made in the study of Fano varieties in characteristic $0$, notably \cite{Bir19, Bir}. Less is known about Fano varieties in positive characteristics.
It is conjectured that Kodaira type vanishing theorems, which in general fail in positive characteristic, hold for Fano varieties of a given dimension in large enough characteristic (\cite[Open questions]{Tot19}).
This conjecture has been confirmed in dimension $2$ by \cite{CTW17}, where the authors prove the existence of an integer $p_0>0$ such that surfaces of del Pezzo type over perfect fields of characteristic $p>p_0$ satisfy the Kawamata--Viehweg vanishing theorem. 
However, an explicit value for $p_0$ cannot be obtained via the methods of \cite{CTW17} and 
in this article we show that $p_0=5$ is both necessary and sufficient:

\begin{thm}[See Theorem \ref{kvv}]\label{t-main-th}
	Let $X$ be a surface of del Pezzo type over a perfect field $k$ of characteristic $p>5$.
	Let $D$ be a Weil divisor on $X$ and suppose that there exists an effective $\Q$-divisor $\Delta$ such that $(X, \Delta )$ is a klt pair and $D-(K_X+\Delta)$ is big and nef. 
	Then 
	$$H^i(X, \mO_X(D))=0 \text{ for all } i>0. $$
\end{thm}

The crucial ingredient in the proof of Theorem \ref{t-main-th} is establishing the following liftability theorem  for klt surfaces whose canonical divisor is not pseudo-effective.

\begin{thm}[See Theorem \ref{liftdelpezzo}]\label{liftdelpezzo-intro} 
	Let $k$ be an algebraically closed field of characteristic $p>5$.
	Let $X$ be a klt projective surface over $k$ such that $K_X$ is not pseudo-effective. 
	Then there exists a log resolution $\mu\colon V\rightarrow X$ such that $(V, \text{Exc}(\mu))$ lifts to characteristic zero over a smooth base.
\end{thm}

In recent years, it has become clear that the understanding of Fano varieties is crucial for the study of klt singularities, in particular in positive characteristics. 
In characteristic zero, Elkik proved that klt singularities are always rational (see \cite{Elk81}). This, however, fails in positive characteristic (see \cite{CT19, Ber, Tot19, Yas19}), and the question is intertwined with the failure of vanishing theorems on Fano varieties. In \cite{HW17} the authors prove that klt threefold singularities are rational in large characteristic $p>p_0$ where $p_0$  is large enough to assure the validity of the Kawamata--Viehweg vanishing theorem for surfaces of del Pezzo type. 
As a consequence of their work and Theorem \ref{t-main-th}, we can therefore provide an effective bound on the characteristic $p$ for which klt threefold singularities are rational.

\begin{cor}[see {\cite[Theorem 1.1]{HW17}}]\label{c-rationality}
	Let $k$ be a perfect field of characteristic $p>5$.
	Let $(X, \Delta)$ be a Kawamata log terminal threefold over $k$.
	Then $X$ has rational singularities.
	Moreover, if $X$ is $\Q$-factorial and $D$ is a Weil divisor on $X$ then $\mO_X(D)$ is Cohen-Macaulay.
\end{cor}

On the other hand Theorem \ref{t-main-th} and Corollary \ref{c-rationality} have several applications to the birational geometry of threefolds. We mention two: the finiteness of the local \'etale fundamental group of threefold klt singularities (see \cite[Theorems A, B and C]{CS}) and a refined version of the base point free theorem for klt threefolds (see \cite{Ber19}). The corollaries below were previously known to hold in large characteristic, we now know them to be true in characteristic $p>5$.

\begin{cor}[See {\cite[Theorems A, B]{CS}}]
	Let $k$ be an algebraically closed field of characteristic $p>5$. Let $(X, \Delta)$ be a Kawamata log terminal threefold over $k$. Let $x\in X$ be a closed point. Then the local \'etale fundamental group of $X$ at $x$ is finite. Moreover, there exists a quasi-\'etale cover $\widetilde{X}\rightarrow X$ with prime-to-$p $ degree such that every quasi-\'etale cover of $\widetilde{X}$ is \'etale.
\end{cor}

\begin{cor}[See  {\cite[Theorem 1.1]{Ber19}}]
	Let $k$ be a perfect field of characteristic $p> 5$. 
	Let $(X, \Delta)$ be a  quasi-projective klt threefold log pair and let $\pi \colon X \to Z$ be a projective contraction morphism of quasi-projective normal varieties over $k$.
	Let $L$ be a $\pi$-nef Cartier divisor on $X$ such that
	\begin{enumerate}
		\item $\dim(Z) \geq 1$ or $\dim(Z)=0$ and $\nu(L) \geq 1$;
		\item $nL-(K_X+\Delta)$ is a $\pi$-big and $\pi$-nef $\Q$-Cartier $\Q$-divisor for some $n>0$.
	\end{enumerate}
	Then there exists $m_0>0$ such that $mL$ is $\pi$-free for all $m \geq m_0$.
\end{cor}

In Section \ref{s-counter-five}, we construct counterexamples in characteristic five to Theorem \ref{t-main-th} and Corollary \ref{c-rationality} building upon a previous example found by the third author.

\begin{thm}[See Proposition \ref{p-failure-vanishing} and Theorem \ref{t-notCM}]
	Let $k$ be an algebraically closed field of characteristic $p=5$.
	Then:
	\begin{enumerate}
		\item there exists a klt del Pezzo surface $T$ of Picard rank one and an ample $\Q$-Cartier divisor $A$ such that $H^1(T, \mathcal{O}_T(-A)) \neq 0$;
		\item there exists a $\Q$-factorial klt threefold singularity which is not Cohen-Macaulay.
	\end{enumerate}
\end{thm}

This, together with the examples found in \cite{CT19} and \cite{Ber}, shows that our assumption on the characteristic in Theorem \ref{t-main-th} and Corollary \ref{c-rationality} is optimal.


\subsection{Sketch of the proof}
We now overview the main techniques and proofs in this article. 
In Section \ref{s-lift} we prove Theorem \ref{liftdelpezzo-intro}. In order
to do this, we first show that we can reduce to the case of such surfaces admitting a Mori fibre space structure (Lemma \ref{p-lift-blow-up}). 
We prove the case of Mori fibre space onto a curve (Proposition \ref{liftmfs}), while the case of klt del Pezzo surfaces of Picard rank one has been settled by the third author in his Ph.D. thesis \cite{Lac20}.

In Section \ref{s-KVV-dP} we use Theorem \ref{liftdelpezzo-intro} to prove Theorem \ref{t-main-th}, where the property of the surface being del Pezzo type plays a fundamental role.
By running a $\Delta$-MMP we construct a birational morphism  $\pi \colon X \to Y$ which contracts a boundary $N\leq \Delta$ such that $D-(K_X+N)$ is ample and such that the support of $N$ coincides with the exceptional locus of $\pi$. Since $Y$ remains a surface of del Pezzo type, we may apply Theorem \ref{liftdelpezzo} to construct a log resolution $(V, \text{Exc}(\mu))$ of $Y$ that lifts to characteristic zero over a smooth base. Since $N$ was contracted by $\pi$, after a sequence of blow-ups of points of $V$, we may find a log resolution of the pair $(X,N)$ lifting to characteristic zero over a smooth base. Then we can prove the desired vanishing theorem by applying the logarithmic Deligne-Illusie vanishing theorem (see Lemma \ref{l-vanishing-W2}). \\

\textbf{Acknowledgements:}  We would like to thank P.~Cascini, C.D.~Hacon, Zs.~Patakfalvi, D.C.~Veniani, J.~Witaszek and  M.~Zdanowicz for useful comments on the article.
We are indebted to M. Zdanowicz for several useful discussions on the topic of liftability to the ring of Witt vectors and for providing key arguments used in the subsection \ref{subsection lift}.
E.A. was supported by SNF Grant \#200021/169639, F.B. was partially supported by the NSF research grant n. DMS-1801851 and by a grant from the Simons Foundation; Award Number: 256202, J.L. was partially supported by the NSF research 
grants no: 1265263 and no: 1802460 and by grant \#409187 from the Simons Foundation.

\section{Preliminaries}

\subsection{Notation}
In this article, $k$ denotes a perfect field of characteristic $p>0$. For $n>0$, we denote by $W_n(k)$ (resp. $W(k)$) the ring of Witt vectors of length $n$ (resp. the ring of Witt vectors).

We say that $X$ is a {\em variety over} $k$ or a $k$-{\em variety} if 
$X$ is an integral scheme that is separated and of finite type over $k$. 
We say that $X$ is a {\em curve} over $k$ or a $k$-{\em curve} 
(resp. a {\em surface} over $k$ or a $k$-{\em surface}, 
resp. a {\em threefold} over $k$) 
if $X$ is a $k$-variety of dimension 1 (resp. 2, resp. 3). 

We say $(X,\Delta)$ is a \emph{log pair} if $X$ is  normal variety, $\Delta$ is an effective $\Q$-divisor and $K_X+\Delta$ is $\Q$-Cartier. 
We say that a pair $(X, \Delta)$ is \emph{log smooth} if $X$ is regular and the support of $\Delta$ is simple normal crossing.
We refer to \cite{Kol13} for the basic definitions in birational geometry and of the singularities appearing in the Minimal Model Program (e.g. klt) and to \cite[Remark 2.7]{GNT19} for birational geometry over perfect fields. We say that a variety is $\mathbb{Q}$-factorial if every Weil divisor is $\mathbb{Q}$-Cartier. We recall that klt surface pairs are $\mathbb{Q}$-factorial (see e.g. \cite[Fact 3.4]{Tan15}), so we will often 
discuss Weil divisors on klt surface pairs without mentioning that they are $\mathbb{Q}$-Cartier.

Given a proper birational morphism $\pi \colon X \to Y$ between normal varieties, we denote by $\text{Exc}(\pi)$ the exceptional locus of $\pi$. Given a variety $X$ and a Weil divisor $D$ we say that $f \colon Y \to X$ is a \emph{log resolution} of $(X,D)$ if $\text{Exc}(\pi)$ has pure codimension one and the pair $(Y, \text{Exc}(\pi)+\pi_*^{-1}D)$ is log smooth.

\subsection{Surfaces of del Pezzo type}\label{del pezzo section}

In this subsection we collect some properties of surfaces of del Pezzo type.

\begin{definition}
	Let $k$ be a field.
	We say that a normal $k$-projective surface $X$ is  a \emph{surface of del Pezzo type} if there exists an effective $\Q$-divisor $\Delta$ such that $(X, \Delta)$ is klt and $-(K_X+\Delta)$ is ample. 
	The pair $(X, \Delta)$ is said to be a log del Pezzo pair. We say $X$ is a \emph{klt del Pezzo surface} if the pair $(X,0)$ is log del Pezzo.
\end{definition}

Surfaces of del Pezzo type behave nicely from the birational perspective, as the MMP algorithm can be run for any divisor on them and they form a closed class under birational contraction.

\begin{lem}\label{l-dP-mds}
	Let $k$ be a perfect field.
	Let $X$ be a surface of del Pezzo type over $k$. 
	Then $X$ is a Mori dream space, i.e. given a $\Q$-divisor $D$ on $X$ we can run a $D$-MMP which terminates.
\end{lem}

\begin{proof}
	Let $\Delta$ be an effective $\Q$-divisor on $X$ such that $(X, \Delta)$ is a log del Pezzo pair.
	Choose a sufficiently small rational number $\varepsilon>0$ such that $A=\varepsilon D-(K_X+\Delta)$ is ample.
	By \cite[Lemma 2.8]{GNT19}, there exists an effective $\Q$-divisor  $A'$ such that $A' \sim_{\Q} A$ and $(X, \Delta+A')$ is a klt pair. 
	We conclude by noting  that a $D$-MMP coincides with a $(K_X+\Delta+A')$-MMP, which exists and terminates by \cite[Theorem 1.1]{Tan18}. 
\end{proof}

\begin{lem}\label{l-dP-contr}
	Let $k$ be a field.
	Let $X$ a surface of del Pezzo type over $k$.
	Let $\pi \colon X \to Y$ be a birational $k$-morphism onto a normal surface $Y$.
	Then $Y$ is a surface of del Pezzo type.
\end{lem}

\begin{proof}
	See \cite[Lemma 2.9]{BT}.
\end{proof}

\subsection{Lifting to characteristic zero over a smooth base}\label{subsection lift}

We now collect some notions on liftability we will use in this article. 
In this subsection we fix $k$ to be an algebraically closed field.
We start by recalling the definition of liftability to characteristic zero (see \cite[Definition 2.15]{CTW17}) and to the ring of Witt vectors. 

\begin{definition}
	Let $(X,D=\sum_{i=1}^r D_i)$ be a log smooth pair over $k$. 
	A \emph{lifting}  of $(X, D)$ to a scheme $T$ consists of
	\begin{enumerate}
		\item a smooth and separated morphism $\mathcal{X}$ over $T$,
		\item effective Cartier divisors $\mathcal{D}_i$ in $\mathcal{X}$ such that for every subset $J \subseteq \left\{1, \dots, r\right\}$ the scheme theoretic intersection $\bigcap_{j \in J} \mathcal{D}_j$ is smooth over $T$,
		\item  a morphism $\alpha \colon \text{Spec}(k)\rightarrow T$,
	\end{enumerate}
	such that the base change of the schemes $\mathcal{X}, \mathcal{D}_1, \dots, \mathcal{D}_r$ via $\alpha$ are isomorphic to $X, D_1, \dots, D_n$ respectively. 
	
	If $T$ is smooth and separated over $\text{Spec}(\Z)$ (resp. $T=\text{Spec}(W(k))$),
	then we say that $(\mathcal{X},\mathcal{D})$ is \emph{a lifting of $(X,D)$ to characteristic zero over a smooth base} (resp. \emph{a lifting to the ring of Witt vectors}).
\end{definition}

We relate the notions of liftability to characteristic zero and liftability to the ring of Witt vectors.

\begin{prop} \label{p-W-char0}
	Let $(X,D)$ be a log smooth pair over $k$.
	Then $(X,D)$ admits a lifting to $W(k)$ if and only if $(X,D)$ admits a lifting to characteristic zero over a smooth base.
\end{prop}

\begin{proof} 
	Suppose that $(X,D)$ admits a lifting $(\mathcal{X},\mathcal{D})$  to characteristic zero over a smooth base $T$. 
	It is sufficient to show that $\alpha \colon \text{Spec}(k) \to T$ lifts to a section $\widetilde{\alpha}\colon \text{Spec}(W(k)) \to T$ since then the base change of $(\mathcal{X},\mathcal{D})$ along $\widetilde{\alpha}$ is a lifting of $(X,D)$ to $W(k)$. 
	The base change $T \times_{\mathbb{Z}} \text{Spec} (W(k)) \to \text{Spec}(W(k))$ is a smooth morphism and $W(k)$ is a complete ring (and thus Henselian). Therefore the section $\alpha \colon \text{Spec}(k) \to T$ lifts to a section $\text{Spec}(W(k)) \to T$ (see \cite[The\'or\`eme 18.5.17]{EGAIV4}). 
	
	Let $(\mathcal{X}, \mathcal{D})$ be a lifting of $(X,D)$ to $W(k)$. 
	As $\text{Frac}(W(k))$ is a field of characteristic zero, $\text{Spec}(\text{Frac}(W(k))$ is geometrically regular at $\mathbb{Q}$ by \cite[\href{https://stacks.math.columbia.edu/tag/038U}{Tag 038U}]{StackProject}.
	This means that $\text{Spec}(W(k))\rightarrow \text{Spec}(\Z)$ is flat with geometrically regular fibers and the morphism $\text{Spec}(W(k))\rightarrow \text{Spec}(\Z)$ is regular (see \cite[\href{https://stacks.math.columbia.edu/tag/07R7}{Tag 07R7}]{StackProject}).
	By a theorem of Popescu (see \cite[Theorem 1.3]{CdJ02}), the ring $W(k)$ is therefore a direct limit of smooth $\Z$-algebras.
	By \cite[Th\'eor\`eme 11.2.6]{EGAIV3} flatness is preserved under spreading out and thus there exists a smooth and separated $\Z$-scheme $T$ such that $\mathcal{X}$ and all possible intersection of $\mathcal{D}_i$ for $1\leq i\leq n$, are defined and smooth over $T$.
	Thus we conclude that $(X,D)$ admits a lifting to characteristic zero over a smooth base.
\end{proof}

\begin{remark} \label{r-lift-W2}
	As a consequence let us note that that if $(X, D)$ admits a lifting to characteristic zero over a smooth base then it admits a lifting over $W_2(k)$ (cf. \cite[Remark 2.16]{CTW17}).
\end{remark}

We are interested in understanding how liftability is preserved under blow-ups at smooth points.
We therefore introduce the following terminology:

\begin{definition}
	Let $(X, D=\sum_{i=1}^n D_i)$ be a log smooth surface pair over $k$.
	Let $x$ be closed point of $X$ and let $J \subseteq \left\{1, \dots, n \right\}$ be the subset of maximal cardinality for which $x \in \bigcap_{j \in J} D_j$ (with the convention that $\bigcap_{j \in \emptyset} D_j=X$).
	Let $(\mathcal{X}, \mathcal{D}=\sum_{i=1}^n \mathcal{D}_i)$ be a lifting of $(X,D)$ to a scheme $T$.
	A \emph{lifting} $\overline{x} \subset \mathcal{X}$ of $x$ is a subscheme of $\mathcal{X}$ smooth over $T$ such that $\overline{x} \times_T \text{Spec}(k)=x$.
	We say a lifting $\overline{x}$ is \emph{compatible with the snc structure of} $(\mathcal{X}, \mathcal{D}) $ if $\overline{x} \subseteq \cap_{j \in J} \mathcal{D}_j$. \end{definition}

By possibly changing the base of the lifting, we may always find a lifting of a point to characteristic zero compatible with the snc structure.

\begin{lem} \label{lift-point}
	Let $(X,D=\sum_{i=1}^n D_i)$ be a log smooth surface pair over $k$.
	Let $x$ be a closed point of $X$.
	Suppose that $(X, D)$ lifts to characteristic zero over a smooth base.
	Then there exists a lifting of $(X,D)$ over a smooth base such that $x$ lifts compatibly with the snc structure.
\end{lem}

\begin{proof}
	Let $(\mathcal{X}, \mathcal{D})$ be a lifting of $X$ over a $\Z$-smooth and separated scheme $T$.
	Then by Proposition \ref{p-W-char0} there exists a lifting $(\mathcal{X}, \mathcal{D})$ to $W(k)$.
	Let $J \subseteq \left\{1, \dots, n \right\}$ be the subset of maximal cardinality for which $x \in \bigcap_{j \in J} D_j$.
	Consider the following diagram:
	$$\begin{tikzcd} 
		\bigcap_{j \in J}D_j \arrow{r} \arrow{d} & \bigcap_{j \in J}\mathcal{D}_j \arrow{r} \arrow{d} & \mathcal{X} \arrow{dl} \\
		\text{Spec}(k) \arrow{r}  \arrow[u, bend left, "s"]& \text{Spec}(W(k)) &.
	\end{tikzcd}$$
	Since $\bigcap_{j \in J}\mathcal{D}_j $ is smooth over $W(k)$, by \cite[The\'or\`eme 18.5.17]{EGAIV4} there exists a lifting $\overline{s} \colon \text{Spec}(W(k)) \to \bigcap_{j \in J} \mathcal{D}_j$. 
	This shows the existence of a lifting of $x$ compatible with the snc structure over $W(k)$.
	By using Popescu's theorem and a spreading out argument as in the proof of Proposition \ref{p-W-char0}, we conclude the existence of a lifting of $(X,D)$ to characteristic zero over a smooth base such that $x$ lifts compatibly with the snc structure.
\end{proof}

We now show that liftability is preserved under blow-ups at smooth centers.

\begin{prop} \label{p-lift-blow-up}
	Let $(X,D=\sum_{i=1}^n D_i)$ be a log smooth surface over $k$ which admits a lifting to characteristic zero over a smooth base.
	Let $x$ be a closed point of $X$ and let $\pi \colon X' \to X$ be the blow-up at the point $x$.
	Then $(X', \pi_*^{-1}D \cup \text{Exc}(\pi))$ admits a lifting to characteristic zero over a smooth base.
\end{prop}

\begin{proof}
	Let us choose $(\mathcal{X},\mathcal{D}=\sum_{i=1}^n \mathcal{D}_i)$ a lifting of $(X,D)$ over a smooth and separated scheme $T$ over $\text{Spec}(\Z)$ such that $x$ lifts compatibly with the snc structure, whose existence is guaranteed by Lemma \ref{lift-point}.
	Let $\overline{\pi} \, \colon \mathcal{X'} \to \mathcal{X}$ be the blow-up along $\overline{x}$. 
	Then $\mathcal{X'}$ is a lifting of $X'$ because $\overline{x}$ is smooth and thus, by \cite[Section 8, Theorem 1.19]{Liu02}, $\mathcal{X'}$ is smooth over $T$ and the exceptional divisor $\mathcal{E}$ is a $\mathbb{P}^1$-bundle over $\overline{x}$. 
	We conclude that $(\mathcal{X'}, \overline{\pi}^{-1}_*\mathcal{D} \cup \mathcal{E})$ is a lifting of $(X', \pi_*^{-1}D \cup \text{Exc}(\pi))$ over $T$.
\end{proof}

Apart from its intrinsic interest, the condition of liftability to characteristic zero over a smooth base is important for its application to vanishing theorems (in fact, for our purposes liftability to $W_2(k)$ would be sufficient). 
We will apply the machinery developed by Deligne and Illusie (see \cite{DI87} and \cite{Har98} for a logarithmic version) to prove vanishing theorems for liftable surfaces in positive characteristic, of which we recall the following special case:

\begin{lem} [(cf. {\cite[Lemma 6.1]{CTW17}})]\label{l-vanishing-W2}
	Let $k$ be a perfect field of characteristic $p>2$. 
	Let $(X, \Delta)$ be a two dimensional projective klt pair over $k$. Suppose that there exists a log resolution $\mu \colon V\to X$
	of $(X, \Delta)$ such that $(V, \text{Exc}(\mu) \cup \mu_*^{-1} (\text{Supp} \Delta))$ lifts to $W_2(k)$. 
	Let $D$ be a $\Z$-divisor on X such that $D - (K_X + \Delta)$ is ample.
	Then, $H^i(X, \mathcal{O}_X(D)) = 0$ for any $i > 0$.
\end{lem}

\section{Liftability of klt surfaces whose canonical divisor is not pseudoeffective} \label{s-lift}

In this section we prove that the minimal resolution along with the exceptional locus of a projective klt pair whose canonical divisor $K_X$ is not pseudo-effective admits a lifting to characteristic zero over a smooth base.
First we prove it for klt surfaces admitting a Mori fibre space structure onto a curve (see Proposition \ref{liftmfs}). 
We then use the classification of klt del Pezzo surfaces of Picard rank one to deduce the general result.

\subsection{Liftability of Mori fibre spaces}

We start by discussing liftability of klt Mori fibre spaces onto a curve. 

\begin{lem}\label{liftruledsurf}
	Let $k$ be an algebraically closed field of characteristic $p>0$.
	Let $C$ be a projective smooth curve and let $\pi \colon X \to C$ be a minimal ruled surface over $k$.
	Let $F_i$ be the fibers over closed points $x_i\in C$ for $1\leq i\leq n$.
	Then $(X, \sum_{i=1}^n F_i)$ admits a lifting to $W(k)$.
\end{lem}

\begin{proof}
	By  \cite[V, Proposition 2.2]{Har77} there exists a vector bundle $E$ of rank $2$ on $C$ such that $\pi \colon X \to C$ is isomorphic to $p \colon \mathbb{P}_C(E) \to C$.
	The curve $C$ admits a formal lifting $\mathcal{C}$ to $W(k)$ because the obstruction to the lifting lies in the $H^2(C, T_C)$ by \cite[Proposition 3.1.5, p. 248]{Ill71}, which vanishes for dimension reasons.
	Analogously, the obstruction for lifting $E$ to a vector bundle $\mathcal{E}$ over $\mathcal{C}$ lies in $H^2(C, \mathcal{E}nd(E))$ by \cite[Proposition 3.1.5, p. 248]{Ill71}, which also vanishes.
	Thus we have a lifting $ \mathbb{P}_{\mathcal{C}} (\mathcal{E}) \to \mathcal{C}$.
	Let $x_1, \dots, x_n$ be the closed points on $C$ such that $\pi^*x_i=F_i$. 
	By formal smoothness, we can lift $x_1, \dots, x_n$ to $\mathcal{C}$ and thus the fibers $F_1, \dots, F_n$ admits a formal lifting $\mathfrak{F}_1, \dots, \mathfrak{F}_n$ over $W(k)$. 
	Therefore $(\mathbb{P}_{\mathcal{C}}(\mathcal{E}), \sum_{i=1}^n \mathfrak{F}_i)$ is a formal lifting of $(X, \sum_{i=1}^n F_i)$ to the formal spectrum $\text{Spf}(W(k))$.
	Since $H^2(X, \mathcal{O}_X)=H^0(X, K_X)=0$ by Serre duality, the pair is algebraizable by \cite[Corollary 8.5.6 and Corollary 8.4.5]{FGAex}.
\end{proof}

\begin{prop}\label{liftmfs} 
	Let $k$ be an algebraically closed field of characteristic $p>0$.
	Let $\pi \colon X \to C$ be a projective morphism such that
	\begin{enumerate}
		\item $X$ is a projective surface over $k$ with klt singularities;
		\item $C$ is a smooth projective curve;
		\item $\pi_*\mathcal{O}_X=\mathcal{O}_C$;
		\item $-K_X$ is $\pi$-ample.	
	\end{enumerate}
	Let  $\mu \colon V\rightarrow X$ be the minimal resolution.
	Then the pair $(V, \text{Exc}(\mu))$ is log smooth and it admits a lifting to characteristic zero over a smooth base. 
\end{prop}

\begin{proof} 
	Let us note that the pair $(V, \text{Exc}(\mu))$ is log smooth by the classification of klt surface singularities (see \cite[Section 3.3]{Kol13}).
	Let us denote by $\pi'\colon V\rightarrow C$ the induced fibration. 
	We run a $K_{V}$-MMP over $C$. 
	We have the following diagram. 
	$$\begin{tikzcd} 
		V\arrow{d}[swap]{\pi'}\arrow{r}{g} &Y\arrow{ld}{f}\\
		C
	\end{tikzcd}.$$
	Note that $f\colon Y \to C$ is a minimal ruled surface. 
	Every irreducible component $E$ of $\text{Exc}(\mu)$ is contained in some fiber $F_i$ of $\pi'$. 
	Let $F_1, \dots F_n$ be the minimal collection of distinct fibers such that $\text{Supp}(\text{Exc}(\mu))\subset \text{Supp}(\sum_{i=1}^n F_i).$ Let $x_i \in C$ be the images of $F_i$ for $i=1, \dots, n$.
	Note that we have the inclusion $\text{Supp}(g(F_i))\subset \text{Supp}(F^Y_{i})$ where $F^Y_{i}$ is the fiber over $x_i$ of $Y\rightarrow C$. 
	The pair $(Y, \sum_{i=1}^n F_i^Y)$ admits a lifting to $W(k)$ by Lemma \ref{liftruledsurf} and by Proposition \ref{p-W-char0} it admits a lifting to characteristic zero over a smooth base.
	Since $ \sum_{i=1}^n F_i^Y$ is a reduced snc divisor and $g \colon V\rightarrow Y$ can be factored as a sequence of blow-ups of smooth points, the pair $(X, \text{Exc}(g)+ \sum_{i=1}^n g_*^{-1}F_i)$ admits a lifting to characteristic zero over a smooth base by Proposition \ref{p-lift-blow-up}. 
\end{proof}

\subsection{General case}

The third author classified all klt del Pezzo surfaces of Picard number one in characteristic $p>3$ in his Ph.D. thesis \cite{Lac20}. 
As a consequence of the classification he was able to prove the existence of a log resolution that lifts to characteristic zero for klt del Pezzo surfaces of Picard one defined over algebraically
closed fields of characteristic $p>5$. This serves as the main ingredient for our results on liftability of surfaces of del Pezzo type.

\begin{thm}[(See {\cite[Theorem 7.2]{Lac20}})]\label{JL}
	Let $k$ be an algebraically closed field of characteristic $p>5$.
	Let $S$ be a klt del Pezzo surface of Picard rank one over $k$. Then there exists a log resolution $\mu\colon V\rightarrow S$ such that $(V,\text{Exc}(\mu))$ lifts to characteristic zero over a smooth base. 
\end{thm}

With the result above, we are now ready to prove the main result of the section.

\begin{thm}\label{liftdelpezzo} 
	Let $k$ be an algebraically closed field of characteristic $p>5$.
	Let $X$ be a klt projective surface over $k$ such that $K_X$ is not pseudo-effective. 
	Then there exists a log resolution $\mu\colon V\rightarrow X$ such that $(V, \text{Exc}(\mu))$ lifts to characteristic zero over a smooth base.
\end{thm}

\begin{proof} 
	Let $g\colon X\rightarrow Y$ be a $K_X$-MMP. 
	There are two cases to distinguish:
	\begin{enumerate}
		\item $Y$ is a klt del Pezzo surface of Picard rank one. 
		Then $Y$ admits a log resolution $\mu \colon V\rightarrow Y$ such that $(V, \text{Exc}(\mu))$ lifts to characteristic zero over a smooth base by Theorem \ref{JL}. 
		\item  $Y$ is a surface admitting a Mori fibre space structure $\pi \colon Y\rightarrow C$ onto a curve $C$. Then $Y$ admits a log resolution $\mu \colon V\rightarrow Y$ such that $(V, \text{Exc}(\mu))$ lifts to characteristic zero over a smooth base by Proposition \ref{liftmfs}.
	\end{enumerate}
	
	In both cases, $Y$ admits a log resolution $\mu \colon V \rightarrow Y$ such that $(V, \text{Exc}(\mu))$ lifts to characteristic zero over a smooth base. 
	Consider the induced rational map $\mu'\colon V \dashrightarrow
	X$ and let $\pi \colon V'\rightarrow V$ be a log resolution of the indeterminacy locus of $\mu'$. 
	Since the morphism 
	$\pi \colon V'\rightarrow V$ can be factored as a composition of blow-ups at points, the pair $(V',\pi_*^{-1}\text{Exc}(\mu)+\text{Exc}(\pi))$ lifts to characteristic zero over a smooth base by Proposition \ref{p-lift-blow-up}.
	The commutative diagram below illustrates the situation:
	$$\begin{tikzcd} X\arrow{d}{g}&(V',\pi_*^{-1}\text{Exc}(\mu)+\text{Exc}(\pi))\arrow{l}[swap]{\tilde{\mu }}\arrow{d}{\pi}\\
		Y&(V,\text{Exc}(\mu))\arrow{l}[swap]{\mu}\end{tikzcd}$$
	We have that the support of $\text{Exc}(\tilde{\mu})$ is contained in the support of $\pi_*^{-1}\text{Exc}(\mu)+\text{Exc}(\pi)$. 
	Therefore $\tilde{\mu} \colon (V',\text{Exc}(\tilde{\mu}))\rightarrow X$ is a log resolution that lifts to characteristic zero over a smooth base.
\end{proof}

\begin{remark}
	One cannot drop the assumption on $K_X$ being not pseudo-effective and extend the previous result to the class of birationally ruled surfaces with klt singularities. 
	Indeed, for every prime $p>0$ Cascini and Tanaka construt examples of klt rational surfaces $Y$ of Picard rank one over $\mathbb{F}_p$  for which no log resolution along with the exceptional divisors admits a lifting to $W_2(k)$ (see \cite[Corollary 3.3]{CT18}). 
	In their examples the canonical divisor $K_Y$ is ample if $p>2$.
\end{remark}

As a direct application of the previous result, we can prove the Kodaira vanishing theorem for Weil divisors for the class of klt surfaces whose canonical divisor is not pseudoeffective.

\begin{cor}\label{KV}Let $k$ be a perfect field of characteristic $p>5$.
	Let $X$ be a projective surface over $k$ with klt singularities such that $K_X$ is not pseudo-effective. 
	Then for an ample Weil divisor $D$ on $X$, we have $H^i(X, \mO_X(K_X+D))=0$ for all $i>0$.
\end{cor}

\begin{proof}
	We can assume $k$ is algebraically closed by passing to the algebraic closure.
	Then $X$ is klt and $(D+K_X)-K_X$ is ample. 
	By Theorem \ref{liftdelpezzo} and Remark \ref{r-lift-W2} $X$ has a log resolution that lifts to $W_2(k)$, therefore the result follow from Lemma \ref{l-vanishing-W2}
\end{proof}

\begin{remark} 
	Let us note that Corollary \ref{KV} was first proven in \cite[Theorem D]{Arv20} using different techniques.
\end{remark}

\section{Kawamata-Viehweg vanishing theorem in characteristic $p>5$}
\label{s-KVV-dP}

Using the results on liftability obtained in Section \ref{s-lift}, we are now ready to prove the main result of this article. 

\begin{thm}\label{kvv}
	Let $X$ be a surface of del Pezzo type over a perfect field $k$ of characteristic $p>5$.
	Let $D$ be a Weil divisor on $X$ and suppose that there exists a boundary $\Delta$ such that $(X, \Delta )$ is a klt pair and $D-(K_X+\Delta)$ is big and nef. 
	Then 
	$$H^i(X, \mO_X(D))=0 \text{ for all } i>0. $$
\end{thm}
\begin{proof}
	By taking the base change to the algebraic closure, we may assume $k$ to be algebraically closed. 
	By possibly perturbing the boundary $\Delta$ we may assume that $D-(K_X+\Delta)=A$ where $A$ is ample. 
	Since $X$ is a Mori dream space by Lemma \ref{l-dP-mds}, we can run a $\Delta$-MMP $g \colon X \to Y$, which terminates with $\Delta_Y:=g_*\Delta$ being a nef $\Q$-Cartier divisor.
	Thus we have
	$\Delta= g^*\Delta_Y+N$, where $N\geq 0$ is effective and exceptional by the negativity lemma \cite[Lemma 3.39]{KM98}.
	So we have $D-(K_X+N)=A+g^*\Delta_Y$ is an ample divisor.
	Moreover, since $N \leq \Delta$, we have $(X, N)$ is klt.
	
	By Lemma \ref{l-dP-contr}, $Y$ is a surface of del Pezzo type and therefore there exists a log resolution $\mu \colon V\rightarrow Y$ such that $(V, \text{Exc}(\mu ))$ lifts to characteristic zero over a smooth base by Theorem \ref{liftdelpezzo}. 
	
	Let us consider the induced rational map $\mu'\colon V \dashrightarrow X$ and let $\pi \colon V'\rightarrow V$ be a sequence of blow-ups at smooth points such that the induced birational map $\widetilde{\mu}\colon V'\dashrightarrow X$ is a morphism.  
	The commutative diagram below illustrates the situation:
	$$\begin{tikzcd} X\arrow{d}{g}&(V',\pi_*^{-1}\text{Exc}(\mu)+\text{Exc}(\pi))\arrow{l}{\widetilde{\mu }}\arrow{d}{\pi}\\
		Y&(V,\text{Exc}(\mu))\arrow{l}{\mu}\end{tikzcd}$$
	We have that the support of $\text{Exc}(\widetilde{\mu})$ and the support of $\widetilde{\mu}^{-1}_*N$ are contained in the support of $\pi_*^{-1}\text{Exc}(\mu)+\text{Exc}(\pi)$. Therefore $\tilde{\mu}\colon (V',\text{Exc}(\tilde{\mu})+\text{Supp}(\tilde{\mu}^{-1}_*N))$ lifts to characteristic zero over a smooth base, and thus to $W_2(k)$ by Remark \ref{r-lift-W2}. 
	By Lemma \ref{l-vanishing-W2}, the result follows. 
\end{proof}

\begin{remark} 
	Motivated by Theorem \ref{liftdelpezzo}, one could ask whether the Kawamata-Viehweg vanishing theorem still holds for the larger class of klt surfaces $X$ whose canonical bundle is not pseudoeffective, at least in large enough characteristic.
	However, this is false in general as shown by the counterexamples in \cite[Theorem 3.1]{CT18} and \cite[Theorem 3.1]{Xie10}.
\end{remark}

\section{Counterexamples in characteristic $p=5$}\label{s-counter-five}

It is well known that the Kawamata-Viehweg vanishing theorem may fail in positive characteristic. 
Counterexamples for klt del Pezzo surfaces were found in \cite{CT19} and \cite{Ber} in characteristic $p=2$ and $p=3$ respectively. 
As a consequence, in these characteristics there are threefold klt singularities which are not rational. 
The third author exhibited an example of a Picard number one klt del Pezzo surface in characteristic $5$ that does not admit a log resolution that lifts to characteristic zero over a smooth base in 
\cite[Example 7.6]{Lac20}. 
In this section we show that this example also violates the  Kawamata-Viehweg vanishing theorem. In particular, Theorem \ref{t-main-th} and Corollary \ref{c-rationality} are sharp. 

\subsection{Construction} \label{ss-construction}

We fix an algebraically closed field $k$ of characteristic $p=5$ and briefly recall \cite[Example 7.6 and B.12]{Lac20}.
Consider the following four points in $\mathbb{P}^2_k$:
\[a=[-1,1,1],\ b=[-1,-1,1],\
c=[1, -1, 1] \text{ and } d=[1,1,1].\] 
Let $L_{ab}$ be the line connecting $a$ and $b$ and similarly define $L_{ac}$, $L_{ad}$, $L_{bc}$, $L_{bd}$ and $L_{cd}$.
Consider the cubic curves $C_0 = L_{ad} + L_{ac} + L_{bc}$ and $C_{\infty} =
L_{ab} + L_{bd} + L_{cd}.$

The base locus of the pencil of curves 
$\left\{ C_t = C_0 + tC_{\infty} \right\}_t$ consists of the points $a$, $b$, $c$, $d$, counted with multiplicity two, and of the point $[0,0,1]$, counted with multiplicity one.
Now resolve the base locus at the points $a$, $b$, $c$, and $d$. This is done in two steps.
First, let $S_1$ be the surface obtained by blowing-up once at $a, b, c$ and $d$. We denote by $E_a$
the exceptional divisor obtained by the first blow-up over $a$, and we use the analogous notation for the other points.
The configuration is illustrated in Figure 1, where the strict transform of $C_0$ is drawn in \textcolor{gray!60}{gray}, the strict transform of $C_\infty$ is drawn in \textcolor{darkgray!95}{dark gray}
and the exceptional divisors are drawn as dotted lines.

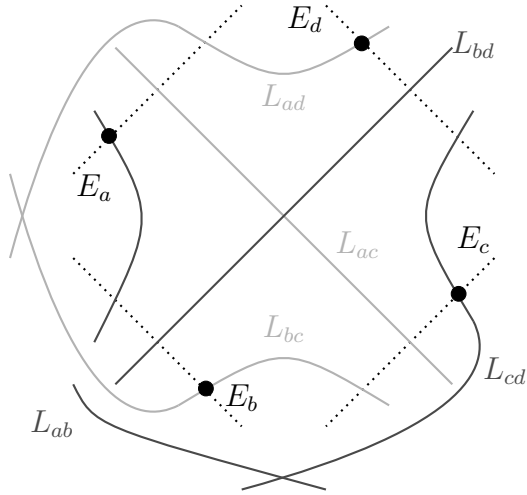
\begin{figure}[h]
	\begin{center}
		\begin{tikzpicture}[scale=.28]
			

			\coordinate (A) at (-10,2);
			\coordinate (B) at (-2,10);
			\coordinate (C) at (2,10);
			\coordinate (D) at (10,2);
			\coordinate (E) at (-10,-2);
			\coordinate (F) at (-2, -10);
			\coordinate (G) at (10, -2);
			\coordinate (H) at (2, -10);
			\begin{scope}[]
				
				
				\draw[draw=black, dotted, thick] (A) -- (B);
				\node [below, black] at (-9, 2.5) {$E_a$};
				\draw[draw=black, dotted, thick] (C) -- (D);
				\node [below, black] at (1,10.5) {$E_d$};
				\draw[draw=black, dotted, thick] (E) -- (F);
				\node [below, black] at (-2,-7.5) {$E_b$};
				\draw[draw=black, dotted, thick] (G) -- (H);
				\node[below, black] at (9, 0) {$E_c$};
				
				
				\draw[draw=gray!60, thick] (-8,8) -- (8,-8);
				\node[above, gray!60] at (3.5, -2.5) {$L_{ac}$};
				\draw[draw= darkgray!95, thick] (-8,-8) -- (8, 8);
				\node[above, darkgray!95] at (9, 7) {$L_{bd}$};
				
				
				\draw[draw=gray!60, thick] (5,9) .. controls (0,6) .. (-5, 9);
				\node[above, gray!60] at (0, 4.5) {$L_{ad}$};
				\draw[draw=gray!60, thick] (-5,9) .. controls (-7,10) and (-10,9) .. (-13,-2);
				\node[above, gray!60] at (0, -6.5) {$L_{bc}$};
				
				\draw[draw=gray!60, thick] (5,-9) .. controls (0,-6) .. (-5, -9);
				\draw[draw=gray!60, thick] (-5,-9) .. controls (-7,-10) and (-10,-9) .. (-13,2);
				
				
				\draw[draw=darkgray!95, thick] (9,5) .. controls (6,0) .. (9,-5);
				\draw[draw=darkgray!95, thick] (9,-5) .. controls (10,-7) and (9,-10) .. (-2,-13);
				\node[above, darkgray!95] at (10.5, -8.5) {$L_{cd}$};
				\draw[draw=darkgray!95, thick] (-9,5) .. controls (-6,0) .. (-9,-6);
				\draw[draw=darkgray!95, thick] (-10,-8) .. controls (-9,-10) .. (2,-13);
				\node[above, darkgray!95] at (-11, -11) {$L_{ab}$};
				
				
				\draw [fill=black] (3.7,8.2) circle [radius=0.35];
				\draw [fill=black] (8.3, -3.7) circle [radius=0.35];
				\draw [fill=black] (-3.7,-8.2) circle [radius=0.35];
				\draw [fill=black] (-8.3, 3.8) circle [radius=0.35];
			\end{scope}
		\end{tikzpicture}
	\end{center}  
	\caption{Configuration of curves on $S_1$ and images of exceptional locus of $\pi \colon S_2 \to S_1$.} \label{Fig1}
\end{figure}

For simplicity, we use the same notation for a curve and its strict transform under a birational morphism. 
Notice that, as drawn in the picture, the lines $L_{ad}$ and $L_{bc}$ meet at infinity, and the same is true for $L_{ab}$ and $L_{cd}$.
Next, let $\pi \colon S_2 \to S_1$ be the surface obtained by blowing up the points 
$E_a \cap L_{ab}$, $E_b \cap L_{bc}$, $E_c \cap L_{cd}$ and $E_d \cap L_{ad}$ (in Figure \ref{Fig1} these points are drawn as black circles). We call $F_a$, $F_b$, $F_c$ and $F_d$ the respective exceptional divisors.
Notice that on $S_2$ the sets $\Gamma_1= \left\{E_a, L_{ad}, L_{bc}, E_c\right\}$ and $\Gamma_2= \left\{ E_d, L_{cd}, L_{ab}, E_b \right\}$ are chains of four $(-2)$ curves. 
Since we started from $\mathbb{P}_k^2$ and performed eight blow ups, we get that $K_{S_2}^2 = K_{\mathbb{P}_k ^2} - 8 = 1$.
The computations in \cite[Example B.12]{Lac20} show that, since the characteristic of $k$ is $5$, there is a rational cuspidal curve $D=C_2$ in the pencil $\left\{C_t\right\} \subset |-K_{\mathbb{P}^2}|$. Note that the strict transform of $D$ is disjoint from $\Gamma_1$ and $\Gamma_2$ and that $D \in |-K_{S_2}|$.
Starting from $S_2$, take a log resolution of the cusp of $D$ by blowing-up three times. We get a surface $V$ with three exceptional rational curves $G_1$, $G_2$, 
$G_3$ and the following intersection numbers:
\[G_1^2=-3,\ G_2^2=-2,\ G_3^2=-1,\ G_1 \cdot G_2=0,\ G_1 \cdot G_3=1,\ G_2 \cdot G_3=1.\]
Note moreover that 
$$	 D^2=-5,\ D\cdot F_a=1,\ D \cdot F_b=1,\\
D\cdot G_1=0,\ D\cdot G_2=0,\ D\cdot G_3=1,\ K_V^2=-2.
$$

The configuration of the curves on $V$ is illustrated in Figure 2. 
We say that a singular point $p$ on a surface is of type $(n)$ if the exceptional divisor of its minimal resolution is a smooth rational curve of degree $-n$.

\begin{figure}[h]
	\begin{center}
		\begin{tikzpicture}[scale=.4]
			
			
			\begin{scope}[]
				
				
				\draw[draw=black, thick] (0,-5) -- (0,5);
				\draw[draw=black, thick] (-2,-2.5) -- (2,-2.5);
				\draw[draw=black, thick] (2,0) -- (-2,0);
				\draw[draw=black, thick] (2,2.5) -- (-2,2.5);
				
				
				\draw[draw=black, dotted, thick] (-7,5) -- (-5, 1);
				\node[above, black] at (-7, 3) {$E_a$};
				\draw[draw=gray!60, thick] (-5,3) -- (-7, -1) ;
				\node[above, gray!60] at (-4.5, 3) {$L_{ad}$};
				\draw[draw=gray!60, thick] (-7,1) -- (-5, -3) ;
				\node[above, gray!60] at (-4.5, -4.5) {$L_{bc}$};
				\draw[draw=black, dotted, thick] (-5, -1) -- ( -7, -5 ) ;
				\node[above, black] at (-7, -4) {$E_c$};
				\node [below] at (-7,-5.2) {$\Gamma_1$};
				
				\draw[draw=black, dotted, thick] (7,5) -- (5, 1);
				\node[above, black] at (7, 3) {$E_b$};
				\draw[draw=darkgray!95, thick] (5,3) -- (7, -1) ;
				\node[above, darkgray!95] at (4.5, 3) {$L_{ab}$};
				\draw[draw=darkgray!95, thick] (7,1) -- (5, -3) ;
				\node[above, darkgray!95] at (4.5, -4.5) {$L_{cd}$};
				\draw[draw=black, dotted, thick] (5, -1) -- (7, -5 ) ;
				\node[above, black] at (7, -4) {$E_d$};
				\node [below] at (7,-5.2) {$\Gamma_2$};
				
				
				\node [below] at (0,-5.2) {$G_3$};
				\node [below] at (1,-2.5) {$G_2$};
				\node [below] at (1,0) {$G_1$};
				\node [below] at (1,2.5) {$D$};
				
			\end{scope}
			
			\draw[->] (9,0) -- (17, 0);
			\node at (13,0) [above] {$\psi$};
			
			\begin{scope}[shift={(23,0)}]
				
				
				\draw[draw=black, thick] (0,-5) -- (0,5);
				
				
				\draw [fill=black] (4,0) circle [radius=0.2];
				\node [below] at (4,-.5) {$A_4$};
				
				\draw [fill=black] (-4,0) circle [radius=0.2];
				\node [below] at (-4,-.5) {$A_4$};	
				
				\node [below] at (0,-5.2) {$G_3^T$};
				\node [below] at (1,-1.7) {$(2)$};
				\draw [fill=black] (0,-2.5) circle [radius=0.2];
				\draw [fill=black] (0,0) circle [radius=0.2];
				\node [below] at (1,.8) {$(3)$};
				\node [below] at (1,3.3) {$(5)$};
				\draw [fill=black] (0,2.5) circle [radius=0.2];
				
			\end{scope}
			
		\end{tikzpicture}
	\end{center}  
	\caption{Arrangement of curves on $V$ and the type of singularities on $T$.} \label{Fig2}
\end{figure}
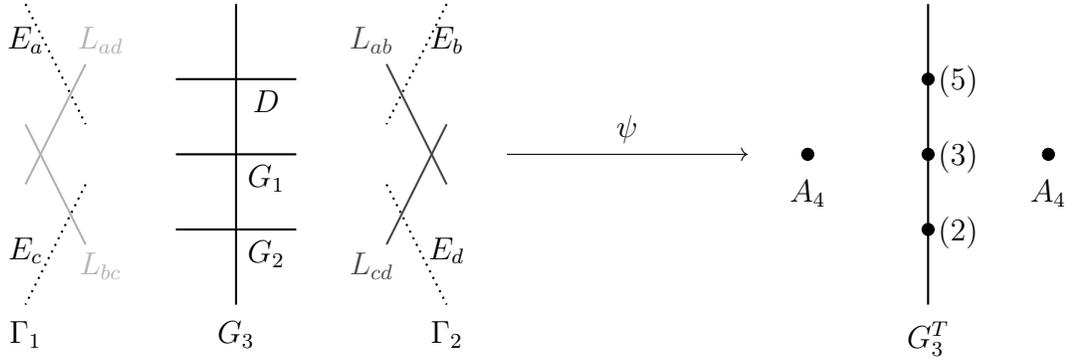

Finally, we consider the birational morphism $\psi \colon V \to T$ which contracts exactly the curves $\Gamma_1, \Gamma_2, G_1, G_2, D$ also illustrated in Figure \ref{Fig2}. 
We denote $F_{a}^T$, $F^{T}_b$, $F^{T}_c$, $F^{T}_d$ and $G_3 ^T$ the image via $\psi$ of $F_a$, $F_b$, $F_c$, $F_d$ and $G_3$ respectively.

\begin{lem}
	$T$ is a surface of Picard number one with only klt singularities.
\end{lem}
\begin{proof}
	The singularities on $T$ are two $A_4$ points, one $A_1$ point, one $(3)$ point and one $(5)$ point, so they are klt.
	As $V$ has Picard rank 12, it is clear that $T$ has Picard rank one.
	Since the Picard rank is one, it is sufficient to compute the intersection of an effective curve with the canonical class to verify $T$ is del Pezzo. By projection formula 
	$$K_T\cdot G^T_{3} = \psi^*(K_T) \cdot G_3 = (K_V+ \frac{1}{3}G_1 + \frac{3}{5}D)\cdot G_3= -1 + 1/3+3/5 <0,$$ as desired.
\end{proof}

\subsection{A counterexample to the Kawamata-Viehweg vanishing theorem} \label{ss-counterexample}

In the following we construct a counterexample to the Kodaira vanishing theorem for a Weil divisor $A$ on $T$.
The choice of the specific Weil divisor was motivated by the fact that it passes through all the singular points and it has maximal Cartier index at each point.

Using the same notations of Subsection \ref{ss-construction}, we start with the following computations: 

\[ \psi^{*}F_a^T= F_a+\frac{4}{5}E_a+\frac{3}{5}L_{ad}+\frac{2}{5}L_{bc}+\frac{1}{5}E_c+\frac{3}{5}E_b+\frac{6}{5}L_{ab}+\frac{4}{5}L_{cd}+\frac{2}{5}E_d+\frac{1}{5}D,\]
\[ \psi^*F_b^T=F_b+\frac{3}{5}E_c+\frac{6}{5}L_{bc}+\frac{4}{5}L_{ad}+\frac{2}{5}E_a+\frac{1}{5}E_d+\frac{2}{5}L_{cd}+\frac{3}{5}L_{ab}+\frac{4}{5}E_b+\frac{1}{5}D, \]

\[\psi^*G_{3}^T=G_3+\frac{1}{3}G_1+\frac{1}{2}G_2+\frac{1}{5}D.\]

\begin{prop}\label{p-failure-vanishing}
	Let $A=G_3^T+F_a^T-F_b^T$. Then $A$ is an ample $\Q$-Cartier Weil divisor and
	\[ H^1(T, \mathcal{O}_T(-A)) \neq 0. \]
\end{prop}

\begin{proof}
	We first verify that $A$ is ample.
	We introduce $B:=-A=F_b^T-F_a^T-G_3^T$
	and we compute its pull-back to $V$:
	\[\psi^*B=F_b-F_a-G_3-\frac{2}{5}E_a-\frac{3}{5}L_{ab}-\frac{2}{5}L_{cd}+\frac{2}{5}E_c-\frac{1}{5}E_d\]
	\[+\frac{1}{5}L_{ad}+\frac{4}{5}L_{bc}+\frac{1}{5}E_b-\frac{1}{3}G_1-\frac{1}{2}G_2-\frac{1}{5}D.\]
	Since $T$ has Picard rank one and $B \cdot F_a^T=\psi^*B \cdot F_a=-\frac{1}{5}$ by projection formula, we conclude $B$ is anti-ample.
	
	We now verify that $H^1(T, \mathcal{O}_T(B)) \neq 0$, using the Riemann-Roch theorem for surfaces. 
	To this end we write
	\[\lfloor{ \psi^*B \rfloor}=F_b-F_a-G_3-E_a-L_{ab}-L_{cd}-E_d-G_1-G_2-D,\]
	and we compute its intersection numbers with $\psi$-exceptional divisors:
	$$ \lfloor{\psi^*B \rfloor} \cdot L_{ab}=0,\lfloor{\psi^*B \rfloor} \cdot L_{bc}=1, \lfloor{\psi^*B \rfloor} \cdot L_{cd}=0,
	\lfloor{\psi^*B \rfloor} \cdot L_{ad}=-1,$$
	$$		\lfloor{\psi^*B \rfloor} \cdot E_a=1,
	\lfloor{\psi^*B \rfloor} \cdot E_b=0, 
	\lfloor{\psi^*B \rfloor} \cdot E_c=0, 
	\lfloor{\psi^*B \rfloor} \cdot E_d=1, $$
	$$		\lfloor{\psi^*B \rfloor} \cdot D=4, 
	\lfloor{\psi^*B \rfloor} \cdot G_1=2, 
	\lfloor{\psi^*B \rfloor} \cdot G_2=1. $$
	As a consequence one sees that $\lfloor{\psi^*B \rfloor}-(K_V+\frac{1}{3}G_1+\frac{3}{5}D+\frac{1}{2}L_{ad})=\lfloor{\psi^*B \rfloor}-(\psi^*K_T+\frac{1}{2}L_{ad})$ is $\psi$-nef and $\psi$-big.
	Thus by the relative Kawamata-Viehweg vanishing theorem for surfaces (see \cite[Theorem 3.3]{Tan18}) we deduce $R^i\psi_*\mathcal{O}_V(\lfloor{\psi^*B \rfloor})=0$ for $i>0$ and therefore $h^i(V, \lfloor{\psi^*B \rfloor})=h^i(T, \mathcal{O}_T(B))$ by the Leray spectral sequence.
	To apply Riemann-Roch we need the self-intersection of $\lfloor{\psi^*B \rfloor}$. To this end we also compute the following intersection numbers:
	$$	\lfloor{\psi^*B \rfloor} \cdot F_a=-2, \lfloor{\psi^*B \rfloor} \cdot F_b=-2, \lfloor{\psi^*B \rfloor} \cdot G_3=-2.
	$$
	A straightforward computation now shows $\lfloor{\psi^*B \rfloor}^2=-7$ and therefore by the Riemann-Roch theorem we conclude
	\[ \chi(V, \lfloor{\psi^*B \rfloor}) = 1+\frac{1}{2}(\lfloor{\psi^*B \rfloor}^2-K_V \cdot \lfloor{\psi^*B \rfloor}) = 1+\frac{1}{2}(-7+3)<0,  \]
	thus showing that $h^1(V, \lfloor{\psi^*B \rfloor})=h^1(T, \mathcal{O}_T(B)) \neq 0$.
\end{proof}

The following corollary shows that the lower bound on the characteristic for the liftability of a log resolution of a klt del Pezzo surfaces to the ring of second Witt vectors in Theorem \ref{liftdelpezzo} is optimal.

\begin{cor}
	The surface $T$ does not admit a log resolution $\mu \colon V \to T$ such that the log pair $(V, \text{Exc}(\mu))$ lifts to $W_2(k)$.
\end{cor}

\begin{proof}
	Since $H^1(T, \mathcal{O}_T(-A))=H^1(T, \mathcal{O}_T(K_T+A)) \neq 0$, the surface $T$ cannot admit a log resolution lifting to $W_2(k)$ by Lemma \ref{l-vanishing-W2}. 
\end{proof}

\subsection{A klt threefold singularity which is not rational}

With the same notation as in Subsection \ref{ss-counterexample}, we consider the ample $\Q$-Cartier Weil divisor $A$ on the Picard rank one klt del Pezzo surface $T$.
Let
\[ X:= C_a(T, \mathcal{O}_T(A))=\text{Spec}_k \bigoplus_{m \geq 0} H^0(T, \mathcal{O}_T(mA)) \]
be the affine cone over $T$ induced by $A$.
Let us denote by $v$ the vertex of the cone, \emph{i.e.} the closed subscheme defined by the ideal $\bigoplus_{m \geq 1} H^0(T, \mathcal{O}_T(mA))$.

\begin{thm}\label{t-notCM}
	The variety $X$ is a normal $\Q$-factorial klt threefold which is not Cohen-Macaulay (and thus not rational).
\end{thm}

\begin{proof}
	By \cite[Proposition 2.4, (2)]{Ber}, $X$ is $\Q$-factorial.
	By inversion of adjunction (see \cite[Corollary 1.5]{HW}), we can apply the proof of \cite[Proposition 2.5, (3)]{Ber} to conclude that $X$ has klt singularities.
	Using \cite[Proposition 2.6]{Ber}, we have
	\[ H^2_v(X, \mathcal{O}_X) \simeq \bigoplus_{m \in \mathbb{Z}} H^{1} (T, \mathcal{O}_T(mA)) \neq 0,  \]
	by Proposition \ref{p-failure-vanishing}.
	Thus we conclude that $X$ is not Cohen-Macaulay.
\end{proof}

\end{document}